\documentclass[11pt]{amsart}
\usepackage{amssymb,a4wide}
\usepackage{color}

\newcommand{\F}{\mathcal F}
\newcommand{\M}{\mathcal M}
\newcommand{\N}{\mathcal N}
\newcommand{\K}{\mathcal K}
\newcommand{\U}{\mathcal U}

\newcommand{\w}{\omega}
\newcommand{\IN}{\mathbb N}

\newcommand{\rank}{\mathsf{rank}}
\newcommand{\conv}{\mathrm{conv}}
\newcommand{\IR}{\mathbb R}
\newcommand{\id}{\mathrm{id}}
\newcommand{\dom}{\mathrm{dom}}
\newcommand{\pr}{\mathrm{pr}}
\newcommand{\St}{\mathrm{St}}
\newcommand{\Sd}{\mathrm{Sd}}

\newcommand{\overbar}[1]{\mkern 1.5mu\overline{\mkern-1.5mu#1\mkern-1.5mu}\mkern 1.5mu}

\newtheorem{lemma}{Lemma}
\newtheorem{theorem}{Theorem}
\newtheorem{proposition}{Proposition}
\newtheorem{claim}{Claim}
\newtheorem{corollary}{Corollary}

\newtheorem{problem}{Problem}

\title{Supercompact minus compact is super}

\author{Taras Banakh, Zdzis\l aw Koszto\l owicz, S\l awomir Turek}

\address{T.~Banakh: Ivan Franko National University of Lviv (Ukraine) and Jan Kochanowski University in Kielce (Poland)}
\email{t.o.banakh@gmail.com}
\address{Z.~Koszto\l owicz: Jan Kochanowski University in Kielce (Poland)}
\email{zkosztolowicz@gmail.com}
\address{S.~Turek: Cardinal Stefan Wyszy\'nski University in Warsaw (Poland)}
\email{s.turek@uksw.edu.pl}

\begin{document}
\begin{abstract} According to a folklore characterization of supercompact spaces, a compact Hausdorff space is supercompact if and only if it has a binary closed $k$-network. This characterization suggests to call a topological space {\em super} if it has a binary closed $k$-network $\mathcal N$. The binarity of $\mathcal N$ means that every linked subfamily of $\mathcal N$ is centered. Therefore, a Hausdorff space is supercompact if and only if it is super and compact. In this paper we prove that the class of super spaces contains all GO-spaces, all supercompact spaces, all metrizable spaces, and all collectionwise normal $\aleph$-spaces. Moreover, the class of super spaces is closed under taking Tychonoff products and discretely dense sets in Tychonoff products. The superness of metrizable spaces implies that each compact metrizable space is supercompact, which was first proved by Strok and Szyma\'nski (1975) and then reproved by Mills (1979), van Douwen (1981), and D\c ebski (1984).
\end{abstract} 
\maketitle

\section{Introduction}

By the classical Alexander Lemma \cite{Ax} (see also \cite[3.12.2]{Eng}), a topological space $X$ is compact if and only if it has a subbase $\mathcal B$ of the topology such that any cover $\mathcal U\subset \mathcal B$ of $X$ contains a finite subcover. However, for the closed interval $[0,1]$, the standard subbase $$\mathcal B=\{(a,1]:0<a<1\}\cup\{[0,b):0<b<1\}$$ of its topology has a much stronger property: each open cover $\U\subset\mathcal B$ of $[0,1]$ contains a two-element subcover $\{[0,b),(a,1]\}\subset\U$. 

This property of the closed interval motivated de Groot \cite{JdG} to introduce supercompact spaces as topological spaces $X$ possessing a subbase $\mathcal B$ of the topology such that every cover $\mathcal U\subset\mathcal B$ of $X$ contains a two-element subcover $\{U_1,U_2\}\subset\U$. The class of supercompact spaces includes all linearly ordered compact space and is closed under Tychonoff products. By a result of Strok and Szyma\'nski \cite{StSz} (reproved in \cite{Mills}, \cite{Douwen}, \cite{Dem}), every compact metrizable space is supercompact. On the other hand, the Stone-\v Cech compactification $\beta\omega$ of the countable discrete space $\omega$ is not supercompact by a known result of Bell \cite{Bell78}. More information on supercompact spaces can be found in the monograph of van Mill \cite{vM}.

There exists a convenient characterization of the supercompactness in terms of binary $k$-networks.  Let us recall that a family $\mathcal N$ of subsets of a topological space $X$ is called 
\begin{itemize}
\item a {\em $k$-network} if for any open set $U\subset X$ and a compact set $K\subset U$ there exists a finite subfamily $\mathcal F\subset\mathcal N$ such that $K\subset\bigcup\mathcal F\subset U$;
\item a {\em closed $k$-network} if $\mathcal N$ is a $k$-network and each set $N\in\mathcal N$ is closed in $X$;
\item {\em linked} if for any sets $A,B\in\mathcal N$ the intersection $A\cap B$ is not empty;
\item {\em centered} if for any finite nonempty subfamily $\mathcal F\subset\mathcal N$ the intersection $\bigcap\mathcal F$ is not empty;
\item {\em binary} if any linked subfamily $\mathcal F\subset\mathcal N$ is centered.
\end{itemize}

For the first time, the following (folklore) characterization of supercompactness was explicitly proved in \cite[2.2]{BKT}.

\begin{theorem}\label{t:super} A compact Hausdorff space $X$ is supercompact if and only if it possesses a binary closed $k$-network.
\end{theorem}

This theorem shows that in the class of compact Hausdorff spaces the supercompactness is equivalent to the existence of a binary closed $k$-network.  So, the latter property can be called the superness. More precisely, we define a topological space $X$ to be {\em super} if it possess a binary closed $k$-network. In this terminology,  Theorem~\ref{t:super} says that a Hausdorff topological space is supercompact if and  only if it is super and compact.

Looking at this characterization it is natural to ask which topological spaces (besides supercompact Hausdorff spaces) are super.

\begin{theorem}\label{t:super} The class of super topological spaces contains all supercompact spaces, all generalized ordered spaces, all metrizable spaces and all collectionwise normal $\aleph$-spaces.
\end{theorem}

Let us recall that a regular topological space $X$ is called 
\begin{itemize}
\item {\em generalized ordered} (briefly, a {\em GO-space}) if $X$ has a base of the topology consisting of order-convex sets with respect to some linear order on $X$;
\item an {\em $\aleph$-space} if $X$ has a $\sigma$-discrete $k$-network;
\item an {\em $\aleph_0$-space} if $X$ possesses a countable $k$-network.
\end{itemize} 
Because of the regularity, in above definitions we can assume that the $k$-networks are closed.
By the Nagata-Smirnov Metrization Theorem \cite[4.4.7]{Eng}, a regular topological space is metrizable if and only if it has a $\sigma$-discrete base. This characterization implies that each (separable) metrizable space is an $\aleph$-space (and an $\aleph_0$-space).

$\aleph_0$-spaces were introduced by Michael \cite{Mich} who proved that the class of $\aleph_0$-spaces is closed under taking the function spaces $C_k(X,Y)$ endowed with the compact-open topology. $\aleph$-spaces were introduced by O'Meara \cite{OM} as non-Lindel\"of generalizations of $\aleph_0$-spaces. $\aleph_0$-spaces and $\aleph$-spaces form two important classes of generalized metric spaces that have many applications in Topology, Topological Algebra and Functional Analysis, see \cite[\S11]{Grue}, \cite{Fog}, \cite{Tan94}, \cite{Tan01}, \cite{Ban15}, \cite{GKKW}.

Theorem~\ref{t:super} will be proved in Section~\ref{s:super}. The most difficult part is to prove that collectionwise normal $\aleph$-spaces are super. This is done in the following theorem that will be proved in Section~\ref{s:main}.

\begin{theorem}\label{t:main} Each collectionwise normal $\aleph$-space possesses a binary $\sigma$-discrete closed $k$-network and hence is super.
\end{theorem}  

Since each $\sigma$-discrete family in a Lindel\"of space is at most countable, Theorem~\ref{t:main} implies

\begin{corollary}\label{c:a0} Each $\aleph_0$-space possesses a binary countable closed $k$-network and hence is super.
\end{corollary}

Since compact metrizable spaces are $\aleph_0$-spaces, Corollary~\ref{c:a0} and Theorem~\ref{t:super} imply the Strok--Szyma\'nski theorem \cite{StSz}.

\begin{corollary}[Strok--Szyma\'nski] Each compact metrizable space is supercompact.
\end{corollary} 

Next, we discuss the stability of the class of super spaces under some operations. The results of Bell \cite{Bell78}, \cite{Bell90} witness that the class of super(compact) Hausdorff spaces is not stable under taking closed subspaces or continuous images. Nonetheless, it is stable under taking Tychonoff products and discretely dense sets in Tychonoff products. A subset $X$ of the Tychonoff product $\prod_{\alpha\in A}X_\alpha$ of topological spaces is called {\em discretely dense} if for any element $x\in \prod_{\alpha\in A}X_\alpha$ and any finite set $F\subset A$ there exists a function $y\in X\subset \prod_{\alpha\in F}X_\alpha$ such that $y{\restriction}_F=x{\restriction}_F$. Observe that $X$ is discretely dense in $\prod_{\alpha\in A}X_\alpha$ if and only if $X$ is dense in the Tychonoff product topology of the spaces $X_\alpha$ endowed with the discrete topologies.

\begin{theorem}\label{t:product} Let $X$ be a discretely dense subspace of the Tychonoff product $\prod_{\alpha\in A}X_\alpha$ of topological spaces. If all spaces $X_\alpha$, $\alpha\in A$, are super, then so is the space $X$.
\end{theorem}

Theorem~\ref{t:product} will be proved in Section~\ref{t:product}. In Sections~\ref{s:complex}--\ref{s:key} we prove some lemmas that will be used in the proof of Theorem~\ref{t:main}, which is the most difficult result of this paper. In the proof of this theorem we use the geometric approach of D\c ebski  \cite{Dem} and some ideas from the unpublished survey \cite{Kub}.

Theorem~\ref{t:main} implies that each metrizable space has a binary $\sigma$-discrete closed $k$-network. This motivates the following open problem:

\begin{problem} Has each metrizable space a binary $\sigma$-discrete base of the topology?
\end{problem}

\section{Simplicial complexes and their barycentric subdivisions}\label{s:complex}

In this section we recall the necessary information on simplicial complexes. By an {\em abstract simplicial complex} we understand a pair $K=(V,E)$ consisting of a set $V$ called the {\em set of vertices} of $K$ and a family $E$ of nonempty finite subsets of $V$ such that together with any set $\sigma\in E$ the family $E$ contains all nonempty subsets of $\sigma$. The finite sets $\sigma\in E$ are called the {\em simplexes} of $K$.

Each vertex $v\in V$ will be identified with the characteristic function $\delta_v\colon V\to \{0,1\}$ of the singleton $\{v\}$ in $V$ (so, $\delta_v^{-1}(1)=\{v\}$). The function $\delta_v$ is an element of the Banach space $\ell^1(V)$ of all functions $f\colon V\to\IR$ such that $\sum_{v\in V}|f(v)|<\infty$. 
Each simplicial complex $K=(V,E)$ will be identified with its {\em geometric realization}
$$\overbar K=\bigcup_{\sigma\in E}\conv(\sigma^{(0)})\subset \ell^1(V)$$
where $\sigma^{(0)}=\{\delta_v\colon v\in\sigma\}$. The geometric realization $\overbar K$ of $K$ is endowed with the strongest topology which coincides with the topology inherited from $\ell^1(V)$ on each geometric simplex $\overbar \sigma=\conv(\sigma^{(0)})$, $\sigma\in E$.
For a simplex $\sigma\in E$ let $\partial \sigma=\bigcup\{\overbar{\tau}\colon\tau\subsetneq\sigma\}$ be the {\em combinatorial boundary} of $\sigma$ and $\sigma^\circ=\overbar{\sigma}\setminus\partial\sigma$ be the {\em combinatorial interior} of $\sigma$. 

In the sequel, geometric realizations of abstract simplicial complexes will be called {\em simplicial complexes}.

For an abstract simplicial complex $K$, its {\em dimension} $\dim(K)$ is defined as $\sup\{|\sigma|-1\colon \sigma\in E\}$. It is well-known that the dimension of $K$ equals the topological dimension of the geometric realization $\overbar{K}$ of $K$.

Let $K=(V,E)$ be an abstract simplicial complex. A simplicial complex $K'=(V',E')$ is called a {\em subcomplex} of $K$ if $V'\subset V$ and $E'\subset E$.

For two abstract simplicial complexes $K_1=(V_1,E_1)$ and $K_2=(V_2,E_2)$ with disjoint sets of vertices their {\em joint} $K_1*K_2$ is the simplicial complex $(V_1\cup V_2, E)$ endowed with the set of simplexes $E=\{\sigma:\exists \sigma_1\in E_1,\;\exists \sigma_2\in E_2$ with $\sigma\subset\sigma_1\cup\sigma_2\}$.
The {\em cone} over a simplicial complex $K$ with vertex $v$ in the join $K*\{v\}$ of $K$ and the singleton $\{v\}$.

\medskip

Now we remind  the necessary information on barycentric subdivisions of simplicial complexes. For a simplex $\sigma\in E$ of a complex $K=(V,E)$, the vector $$b_\sigma=\sum_{v\in \sigma}\tfrac{1}{|\sigma|}\delta_v\in\bar K\subset \ell^1(V)$$ is called the {\em barycenter} of $\overbar{\sigma}$.
Let $\sigma\in E$ and $\prec$ be a linear order on the set $\sigma$. By $\overbar{\sigma}_\prec$ we denote the set of all convex combinations $\sum_{v\in\sigma}t_v\delta_v\in\ell^1(V)$ such that  $t_u\ge t_v$ for every vertices $u\prec v$ of the symplex $\sigma$. Observe that  $\overbar{\sigma}_\prec$ is convex hull of the set $\{b_{\{v_1\}},b_{\{v_1,v_2\}},\dots,b_{\{v_1,v_2,\dots,v_k\}}\}$, where $\sigma=\{v_1,v_2,\dots,v_k\}$ and $v_1\prec v_2\prec\dots\prec v_k$.
The family 
$$\Sd(\overline{\sigma})= \{\overbar{\sigma}_\prec\colon \prec \text{ is total order on }\sigma\}$$ consists of $|\sigma|!$ sets and covers the entire set $\overbar \sigma$;  we call it the {\em barycentric subdivision} of $\overbar \sigma$. 
The {\em barycentric subdivision} $\Sd(\overbar K)$ of the complex $\overbar K$
is the family $\bigcup\{\Sd(\overbar\sigma)\colon \sigma\in E\}$.
Let us consider any $\overline{\sigma}_\prec\in\Sd(\overbar K)$.  Analogously as above, for each total order $\prec'$ on the set $\{b_{\{v_1\}},b_{\{v_1,v_2\}},\dots,b_{\{v_1,v_2,\dots,v_k\}}\}$, where $\sigma=\{v_1,v_2,\dots,v_k\}$ and $v_1\prec v_2\prec\dots\prec v_k$ we can consider the set $\overbar{\sigma}_{\prec,\prec'}\subset \ell_1(V)$ consisting of all convex combinations of the form
$$\sum_{i=1}^k t_i b_{\{v_1,\dots,v_i\}}  $$ 
such that if $b_{\{v_1,\dots,v_i\}}\prec' b_{\{v_1,\dots,v_j\}}$ then $t_i\ge t_j$.  The family of all $\overbar{\sigma}_{\prec,\prec'}$ forms the {\em second barycentric subdivision} $\Sd^2(\overbar{\sigma}_\prec)$ of $\overbar{\sigma}_\prec$. The union $$\Sd^2(\overbar{\sigma}):=\bigcup\{\Sd^2(\overbar{\sigma}_\prec)\colon \prec \text{ is a total order on }\sigma\}$$ is called the {\em second barycentric subdivision} of $\overbar{\sigma}$. The family $\Sd^2(\overbar{K})=\bigcup_{\sigma\in E}\Sd^2(\overbar{\sigma})$  is the {\em second barycentric subdivision} of  the complex $\overbar{K}$.

For a simplex $\sigma\in E$ and vertex $u\in\sigma$,  {\em the star} of $\delta_u$ with respect to $\Sd(\overbar\sigma)$ is the set
$$\St^1(\delta_u,\sigma)={\textstyle\bigcup}\{F\in\Sd(\overbar\sigma)\colon \delta_u\in F\}=\Big\{\sum_{v\in\sigma}t_v\delta_v\in\bar K\colon t_u=\max_{v\in\sigma} t_v \Big\}.$$
The union 
$$\St^1(\delta_u,K)={\textstyle\bigcup}\{\St(\delta_u, \sigma)\colon u\in\sigma\text{ and } \sigma\in E\}$$ is called {\em the star} of $\delta_u$ with respect to $\Sd(\overbar K)$.

For a simplex $\tau\in E$ of the complex $K$ by $\St^2(b_\tau,K)$ we denote the closed star of $b_\tau$ in the second barycentric subdivision of $K$, i.e.
$$\St^2(b_\tau,K)={\textstyle\bigcup}\{\overbar{\sigma}_{\prec,\prec'}\in\Sd^2(\overbar{K})\colon b_\tau\in \overbar{\sigma}_{\prec,\prec'}\}.$$ 

\begin{lemma}\label{l1}If $K=(V,E)$ is a simplicial complex and $\tau\in E$, then
\begin{multline*}\St^2(b_\tau,K)=\Big\{\sum_{i=1}^nt_{i} b_{\sigma_i}\colon 
n\ge |\tau|, \sum_{i=1}^nt_i=1, \;
\sigma_1\subset\sigma_2\subset\dots\subset \sigma_n\text{ are in }E,\\ (\forall i\;\; |\sigma_i|=i,\; t_i\ge 0), \;\sigma_{|\tau|}=\tau, \;\;
t_{|\tau|}=\max_{1\le i\le n} t_i\Big\}.
\end{multline*}
\end{lemma}
\begin{proof}
By the definition,
$$\St^2(b_\tau, K)=\bigcup\{\overbar{\sigma}_{\prec,\prec'}\in\Sd^2(\overbar{K})\colon b_\tau\in \overbar{\sigma}_{\prec,\prec'}\}.$$
Fix any $\bar\sigma_{\prec,\prec'}\in \St^2(b_\tau,K)$ and observe that $b_\tau\in	\overbar{\sigma}_{\prec,\prec'}$ where $\sigma=\{v_1,\dots,v_n\}\in E$, $v_1\prec\dots\prec v_n$ and $\prec'$ is linear order on the set $\{b_{\{v_1\}},\dots,b_{\{v_1,\dots,v_n\}}\}$. By the definition of $\overbar{\sigma}_{\prec,\prec'}$ we have
$$b_\tau=\sum_{i=1}^n s_i b_{\{v_1,\dots,v_i\}}$$
where $\sum_{i=1}^n s_i=1$, $s_i\ge 0$ and $s_i\ge s_j$ whenever $b_{\{v_1,\dots,v_i\}}\prec'b_{\{v_1,\dots,v_j\}}$. 
Hence
\begin{multline*}
b_\tau=\tfrac{1}{k}\delta_{u_1}+\dots+\tfrac{1}{k}\delta_{u_k}=\left(s_1+\tfrac{s_2}{2}+\dots+\tfrac{s_n}{n}\right)\delta_{v_1}+\left(\tfrac{s_2}{2}+\dots+\tfrac{s_n}{n}\right)\delta_{v_2}+\dots\\\dots+\left(\tfrac{s_k}{k}+\dots+\tfrac{s_n}{n}\right)\delta_{v_k}+\left(\tfrac{s_{k+1}}{k+1}+\dots+\tfrac{s_n}{n}\right)\delta_{v_{k+1}}+\dots+\tfrac{s_n}{n}\delta_{v_n},\tag{$*$}\end{multline*}
where $\tau=\{u_1,\dots,u_k\}$. By the linear independence of the set $\{\delta_u\colon u\in V\}$ in the space $\ell_1(V)$ we get $\tau\subseteq \sigma$. Moreover, the equality ($*$) implies that $\tau=\{v_1,\dots,v_k\}$, $k\le n$, $s_k=1$ and $s_i=0$ for for any $i\ne k$ in $\{1,\dots,n\}$. Therefore the barycenter $b_{\{v_1,\dots,v_k\}}$ is the least element of the set $\{b_{\{v_1\}},\dots,b_{\{v_1,\dots,v_n\}}\}$ with respect to order $\prec'$. So, any point $x\in\overbar{\sigma}_{\prec,\prec'}$ can be written as 
$$x=\sum_{i=1}^nt_{i} b_{\sigma_i}$$
where $n=|\sigma|\ge|\tau|$, $\sum_{i=1}^nt_i=1$, $t_i\ge 0$, $\sigma_i=\{v_1,\dots,v_i\}$ for each $i=1,\dots, n$, $\sigma_{|\tau|}=\sigma_k=\tau$ and $t_k=\max_{1\le i\le n} t_i$.

Now let $x=\sum_{i=1}^nt_{i} b_{\sigma_i}$ where 
$n\ge |\tau|$, $\sum_{i=1}^nt_i=1$, 
$\sigma_1\subset\sigma_2\subset\dots\subset \sigma_n$  are in $E$,  $(\forall i$\;\; $|\sigma_i|=i$, $t_i\ge 0$), $\sigma_{|\tau|}=\tau$,  and 
$t_{|\tau|}=\max_{1\le i\le n} t_i$. We can assume that $\sigma_i=\{v_1,\dots,v_i\}$ for each $i\in\{1,\dots,n\}$ and  order the set $\sigma=\sigma_n=\{v_1,\dots,v_n\}$ in the following way: $v_1\prec v_2\prec\dots\prec v_n$. On the set  
$\{b_{\{v_1\}},\dots,b_{\{v_1,\dots,v_n\}}\}$ we choose a linear order $\prec'$ such  that 
$b_{\sigma_i}\prec' b_{\sigma_j}$ if $t_i\ge t_j$.
Then $b_\tau\in \overbar{\sigma}_{\prec,\prec'}$ and obviously $x\in \overbar{\sigma}_{\prec,\prec'}$.
\end{proof}

\begin{lemma}\label{l2}
If $K=(V,E)$ is a simplicial complex and $\tau\in E$, then
$$
\begin{aligned}
\St^2(b_\tau,K)=\Big\{\sum_{k=1}^ns_{k} \delta_{v_k}\colon & |\{v_1,\dots,v_n\}|=n\ge |\tau|,\; s_1\ge  \dots\ge s_n\ge 0,\;
\sum_{k=1}^n k(s_k-s_{k+1})=1,\\
&s_{n+1}= 0,\;\;|\tau|(s_{|\tau|}-s_{|\tau|+1}) =
\max_{1\le k\le n} k(s_k-s_{k+1}),\\
& \{v_1,\dots,v_n\}\in E,\;\text{ and } \tau=\{v_1,\dots,v_{|\tau|}\} \Big\}.
\end{aligned}
$$
\end{lemma}	
	\begin{proof}
Let $x\in \St^2(b_\tau,K)$. By the Lemma~\ref{l1}, $x=\sum\limits_{i=1}^nt_i b_{_{\sigma_i}}$, where $\sum\limits_{i=1}^nt_i=1$,	$\sigma_1\subset\dots\subset \sigma_n$  are in $E$, $|\sigma_i|=i$ for each $i=1,\dots,n$, $\sigma_{|\tau|}=\tau$ and $t_{|\tau|}=\max_{1\le i\le n}t_i$. If we  assume that $\sigma_{i}=\{v_1,\dots,v_i\}$ for each $i=1,\dots,n$ then	
$$
b_{\sigma_i}=\tfrac{1}{|\sigma_i|}\sum_{k=1}^{|\sigma_i|}\delta_{v_k}.$$
Thus 
\begin{equation}
x=\sum_{i=1}^nt_i\tfrac{1}{|\sigma_i|}\sum_{k=1}^{|\sigma_i|}
\delta_{v_k}=\sum_{k=1}^n\delta_{v_k}\Big(\sum_{i=k}^n\tfrac{1}{i}t_i\Big)=\sum_{k=1}^ns_k\delta_{v_k},
\tag{$*$}\end{equation}
where
\begin{equation}
s_1=\tfrac{1}{1}t_1+\tfrac{1}{2}t_2+\dots+\tfrac{1}{n}t_n,\;\;
s_2=	\tfrac{1}{2}t_2+\dots+\tfrac{1}{n}t_n,\;\;\dots,\;\;s_n=\tfrac{1}{n}t_n.\tag{$**$}\end{equation}

Obviously $s_1\ge s_2\ge\dots\ge s_n$ and taking  $s_{n+1}=0$ we get 
$$\sum_{k=1}^nk(s_k-s_{k+1})=s_1+s_2+\dots+s_n+s_{n+1}=\sum_{i=1}^nt_i=1.$$
For each $k\in\{1,\dots,n\}$ we have 
$k(s_k-s_{k+1})=k\frac{1}{k}t_k=t_k$, and
consequently
$$|\tau|(s_{|\tau|}-s_{|\tau|+1})=t_{|\tau|}=\max_{1\le k\le n}t_k= \max_{1\le k\le n}k(s_k-s_{k+1}).$$

In order to obtain reverse inclusion, fix  
$x=\sum_{k=1}^ns_{k} \delta_{v_k}$, where   $n\ge |\tau|$, $s_1\ge  \dots\ge s_n\ge 0$,
$\sum_{k=1}^n k(s_k-s_{k+1})=1$, $s_{n+1}= 0$, 
$|\tau|(s_{|\tau|}-s_{|\tau|+1}) =
\max_{1\le k\le n} k(s_k-s_{k+1})$, $\sigma_k=\{v_1,\dots,v_k\}\in E$ for each $k=1,\dots,n$ and $\tau=\sigma_{|\tau|}$. Let 
$t_k=k(s_k-s_{k-1})$ for each $k=1,\dots,n$. Then $\sum_{k=1}^n t_k=1$, $t_{|\tau|}=\max_{1\le k\le n}t_k$ and the system of equations ($**$) is satisfied. Therefore, $x$ can be expressed as in equation ($*$). It means that 
$$x=\sum_{k=1}^{n}t_kb_{\sigma_k},$$
and $x\in\St^2(b_\tau,K)$ by Lemma~\ref{l1}.
		\end{proof}

\begin{lemma}\label{l3}
	If $\St^2(b_\sigma,K)\cap \St^2(b_\tau,K)\ne\emptyset$, then $\sigma$ and $\tau$ are 
	comparable.
\end{lemma}
\begin{proof}
Let $x\in\St^2(b_\sigma,K)\cap\St^2(b_\tau,K)$. We can assume that $|\sigma|\le |\tau|$. We will prove that $\sigma\subset \tau$. By Lemma~\ref{l2},
\begin{equation}
x=\sum_{k=1}^ns_{k} \delta_{v_k}=\sum_{k=1}^mt_{k} \delta_{u_k},\tag{$*$}
\end{equation}
where $\sigma=\{v_1,\dots,v_{|\sigma|}\}$, $|\{v_1,\dots,v_n\}|=n\ge |\sigma|$, $\tau
=\{u_1,\dots,u_{|\tau|}\}$, $|\{u_1,\dots,u_m\}|=m\ge |\tau|$,  
$$s_1\ge\dots\ge s_n\ge  0,\;
\sum_{k=1}^n k(s_k-s_{k+1})=1,\;s_{n+1}=0,\; |\sigma|(s_{|\sigma|}-s_{|\sigma|+1}) =
\max_{1\le k\le n} k(s_k-s_{k+1})$$ 
and  $$t_1\ge \dots\ge t_m\ge  0,\;
\sum_{k=1}^m k(t_k-t_{k+1})=1, t_{m+1}=0,\;  |\tau|(t_{|\tau|}-t_{|\tau|+1}) =
\max_{1\le k\le m} k(t_k-t_{k+1}).$$
We can assume that $s_k>0$ and $t_m>0$ for all $k\le n$ and $k\le m$ because $s_{|\sigma|},t_{|\tau|}>0$.

The linear independence of the set $\{\delta_v\colon v\in V\}$ and the equality ($*$) imply that $n=m$ and  
there exists a permutation $f\colon\{1,\dots,n\}\to\{1,\dots,m\}$ such that $t_{f(i)}=s_i$ and $u_{f(i)}=v_i$ for each $i\in\{1,\dots,n\}$. 
Denote the set $\{s_1,\dots,s_m\}=\{t_1,\dots,t_m\}$ by $C$ and write it as $C=\{c_1,\dots,c_l\}$ for some positive real numbers $c_1>\dots>c_l$. For every $c\in C$ consider the subsets $S_c=\{i\in\{1,\dots,n\}:s_i=c\}$ and $T_c=\{j\in\{1,\dots,m\}:t_j=c\}$. The equality $(*)$ implies that $f(S_c)=T_c$ for any $c\in C$. 
For two non-empty subsets $A,B\subset\mathbb N$ we write $A<B$ if $\max A<\min B$. 
Taking into account that $s_1\ge\dots\ge s_n$ and $t_1\ge\dots\ge t_m$, we conclude that $S_{c_1}<S_{c_2}<\dots<S_{c_l}$ and $T_{c_1}<T_{c_2}<\dots<T_{c_l}$, and hence $S_{c}=T_c$ for all $c\in C$.

Find $j\le l$ with $s_{|\sigma|}=c_j$. Taking into account that $s_{|\sigma|+1}<s_{|\sigma|}$, we conclude that $\{s_1,\dots,s_{|\sigma|}\}=\{c_1,\dots,c_j\}$ and $\{1,\dots,{|\sigma|}\}=S_1\cup\cdots \cup S_j=T_1\cup\cdots \cup T_j$. Then
\begin{multline*}
\sigma=\{v_1,\dots,v_{|\sigma|}\}=\bigcup_{i=1}^j\{v_k: k\in S_i\}=\bigcup_{i=1}^j\{u_{f(k)}: k\in S_i\}=\\
=\bigcup_{i=1}^j\{u_k:k\in T_i\}=\{u_1,\dots,u_{|\sigma|}\subset\{u_1,\dots,u_{|\tau|}\}=\tau.
\end{multline*}
\end{proof}


The last lemma  implies the following lemma.

\begin{lemma}\label{l4}
If $\mathcal{S}$ is a family of simplexes in a simplicial complex $K$  and $\{\St^2(b_\sigma,K)\colon \sigma\in\mathcal{S}\}$ is linked, then $\mathcal{S}$ is a chain.\qed 
\end{lemma}

It turns out that Lemma~\ref{l4} can be reversed in the following sense.

\begin{lemma}\label{l5}
	If $\sigma_1\subset\dots\subset \sigma_n$ is a chain of simplexes, then
	$$\sum_{i=1}^n\tfrac{1}{n}b_{\sigma_i}\in \bigcap_{i=1}^n\St^2(b_{\sigma_i},K).$$
\end{lemma}
\begin{proof}
The chain	$\sigma_1\subset\dots\subset \sigma_n$ can be enlarged to  a chain $\tau_1\subset\dots\subset \tau_m$ such that $|\tau_i|=i$ and $\tau_m=\sigma_n$. For each $i\in\{1,\dots,m\}$ let $t_i=\frac{1}{n}$ if $\tau_i=\sigma_j$ for some $j\in\{1,\dots,n\}$ and $t_i=0$ otherwise. By the Lemma~\ref{l1}
$$\sum_{j=1}^n\tfrac{1}{n}b_{\sigma_j}=\sum_{i=1}^mt_ib_{\tau_i}\in \St^2(b_{\sigma_k},K)$$
for each $k\in\{1,\dots,n\}$.
\end{proof}	
\begin{lemma}\label{l6}
	If $\sigma,\tau\in E$ and $\overbar{\sigma}\cap \St^2(b_\tau,K)\ne\emptyset$ then $\tau\subset\sigma$. In particular $b_\tau\in \overbar{\sigma}$. 
\end{lemma}
\begin{proof}
Let $x\in \overbar{\sigma}\cap  \St^2(b_\tau,K)$.	
By Lemma~\ref{l2}
$$x=\sum_{i=1}^k s_i\delta_{v_i}=\sum_{i=1}^nt_i\delta_{u_i}$$
where $k=|\sigma|$, $\sigma=\{v_1,\dots,v_k\}$, $\sum_{i=1}^ks_i=1$, $|\{u_1,\dots,u_n\}|=n\ge m=|\tau|$, $\tau=\{u_1,\dots,u_m\}$,  $t_1\ge\dots\ge t_n\ge 0$, $\sum_{i=1}^n i(t_i-t_{i+1})=1$, $t_{n+1}=0$, $m(t_m-t_{m+1})=\max_{1\le i\le n}i(t_i-t_{i+1})>0$. Suppose $u_i\notin \sigma$ for some $i\in\{1,\dots,m\}$. The linear independence of the set $\{\delta_v\colon v\in V\}$ implies that $t_i=0$. Then $t_j=0$ for $j\ge i$. In particular $t_m=0$ and hence $m(t_m-t_{m+1})=0$, which is a contradiction.	Thus $\tau\subset \sigma$ and hence $b_\tau\in\overbar{\sigma}$.
\end{proof}

\begin{proposition}\label{subdiv} Let $K$ be a finite-dimensional simplicial complex, $\F$ be a family of subcomplexes of $K$, $C$ be a subcomplex of $K$, and $\mathcal G$ be the family of closed stars of the second barycentric subdivision of $\overbar{C}$ centered at vertices of the first barycentric subdivision of $\overbar C$. Then
\begin{enumerate}
\item $\overbar C=\bigcup\mathcal G$;
\item the family $\mathcal G$ is $n$-discrete where $n=1+\dim(\overbar C)$;
\item each linked subfamily $\mathcal L\subset\F\cup\mathcal G$ with $\mathcal L\cap\mathcal G\ne\emptyset$ has $\bigcap\mathcal L\ne\emptyset$.
\end{enumerate}
\end{proposition}

\begin{proof} 
Let $n=1+\dim(\overbar C)$. Observe that 
$\mathcal{G}=\bigcup\limits_{i=1}^{n}\mathcal{G}_i$, where
	$$\mathcal{G}_i=\{\St^2(b_\tau,C)\colon \tau \text{ is a simplex of  $C$ with } |\tau|=i \}$$
	for each $i\le n$. 
	 Obviously $\overbar{C}=\bigcup\mathcal{G}$.
	We claim that each family $\mathcal{G}_i$ is discrete.
	Indeed, let $\sigma,\tau$ be two distinct simplexes of $C$ with $|\sigma|=|\tau|$. 
		 If $x\in \St^2(b_\sigma,C)$ and $y\in \St^2(b_\tau,C)$ then by Lemma~\ref{l2} we have 
	$$x=\sum_{k=1}^ls_{k} \delta_{v_k}\text{ and } y=\sum_{k=1}^mt_{k} \delta_{u_k},$$
	where 
	\begin{enumerate}
		\item[(i)] $|\{v_1,\dots,v_l\}|=l\ge|\sigma|,\; |\{u_1,\dots,u_m\}|=m\ge|\tau|$,
		\item[(ii)] $s_1\ge\dots\ge s_l\ge s_{l+1}=0$, $t_1\ge\dots\ge t_m\ge t_{m+1}=0$,
		\item[(iii)] $\sum\limits_{k=1}^l k(s_k-s_{k+1})=1$, $\sum\limits_{k=1}^m k(t_k-t_{k+1})=1$,
		\item[(iv)] $|\sigma|(s_{|\sigma|}-s_{|\sigma|+1}) =
		\max\limits_{1\le k\le l} k(s_k-s_{k+1})$, \ $|\tau|(t_{|\tau|}-t_{|\tau|+1}) =
		\max\limits_{1\le k\le m} k(t_k-t_{k+1})$,
		\item[(v)] $\sigma=\{v_1,\dots,v_{|\sigma|}\}$, $\tau=\{u_1,\dots,u_{|\tau|}\}$.
		\item[(vi)] $\{v_1,\dots,v_l\},\{u_1,\dots,u_m\}\in E$ and hence $\max\{l,m\}\le1+\dim(K)=n$.
	\end{enumerate}
Since $\sigma\ne\tau$, there is $j$ such that $v_j\notin\tau$.  Then	
	 	$$||x-y||\ge s_j\ge s_{|\sigma|}\ge \tfrac{1}{|\sigma|}\max\limits_{1\le k\le l} k(s_k-s_{k+1})\ge\frac{1}{l\cdot |\sigma|}\ge \frac{1}{n^2}.$$
Then no ball of radius $\frac1{3n^2}$ in the Banach space $\ell^1(V)$ meets two distinct sets in the family $\mathcal G_i$, witnessing that this family if discrete.
	 Thus
	 $\mathcal{G}=\bigcup\limits_{i=0}^{n}\mathcal{G}_i$ is $n$-discrete. 
	 
We now turn to the proof of condition (3).	 
	Let $\mathcal L\subset\F\cup\mathcal G$, where $\mathcal{L}\cap \mathcal{G}\ne\emptyset$, be a linked family. Let $\mathcal{S}=\{\sigma\in C\colon \St^2(b_\sigma,C)\in \mathcal{L}\cap\mathcal{G}\}$.
By the Lemma~\ref{l4}, the family $\mathcal{S}$ is a chain. So, $\mathcal{S}=\{\sigma_1,\dots,\sigma_k\}$, where $\sigma_1\subset\dots\subset \sigma_k$ and  $\mathcal{L}\cap\mathcal{G}=\{\St^2(b_{\sigma_i},C)\colon 1\le i\le k\}$. By the Lemma~\ref{l5},  
$$\sum_{i=1}^k\tfrac{1}{k}b_{\sigma_i}\in\bigcap_{i=1}^k \St^2(b_{\sigma_i},C).$$

We claim that $\sum_{i=1}^k\frac1k b_{\sigma_i}\in\bar F$ for any subcomplex $\overbar{F}\in\mathcal{L}\cap\mathcal{F}$.
Since  
$\overbar{F}\cap \St^2(b_{\sigma_k},C) \ne\emptyset$, there exists a simplex $\tau$ of $F$ such that $\overline{\tau}\cap \St^2(b_{\sigma_k},C)\ne \emptyset$. Lemma~\ref{l6} implies that $\sigma_k\subset \tau$ and consequently $\sigma_i\subset \tau$ and hence $b_{\sigma_i}\in\bar\sigma_i\subset\bar\tau$ for each $i\in\{1,\dots,k\}$. Then  $\sum_{i=1}^k\frac{1}{k}b_{\sigma_i}\in\overbar{\tau}\subset\overbar{F}$ by the convexity of the geometric simplex $\bar\tau$, and finally $\sum_{i=1}^k\frac1kb_{\sigma_i}\in\bigcap\mathcal L$.
\end{proof}

\section{Simplicial realizations of $n$-discrete families}\label{s:real}

In this section we prove a technically difficult Proposition~\ref{real} on reflexions of $n$-discrete families in simplicial complexes. We start with the following lemma, which is a modification of the Sweeping Out Theorem 1.10.15 in \cite{End}.

\begin{lemma}\label{reduce} Let $L$ be a finite-dimensional simplicial complex, $K\subset L$ be a subcomplex in $L$. For any subset $A\subset \bar L$ there exists a simplicial subcomplex $M$ of $L$ containing $K$ and a continuous map $r:U\to \bar M$ defined on an open neighborhood $U\subset \bar L$ of $\overbar{M}\cup A$ in $\bar L$ such that $r\restriction\overbar{M}=\id$, $r(A\cup\bar K)=\bar M$, and $r(\bar\sigma\cap U)\subset\bar \sigma$ for every simplex $\sigma$ of $L$.
\end{lemma}

\begin{proof} By induction we shall construct a decreasing sequence $(L_n)_{n\in\w}$ of subcomplexes of $L$, a sequence $(A_n)_{n\in\w}$ of subsets $A_n\subset \bar L_n$ and a sequence $(r_{n})_{n\in\w}$ of continuous maps $r_{n}:\dom(r_{n})\to \bar L_{n+1}$ defined on open subsets $\dom(r_{n})$ of $\bar L_n$ such that $\bar L_0=L$, $A_0=A$ and for every $n\in\w$ the following conditions are satisfied:
\begin{itemize}
\item[(a)] $A_{n+1}=r_{n}(A_n)$;
\item[(b)] $K\subset L_{n+1}$ and  $\bar L_{n+1}\cup A_n\subset \dom(r_{n})\subset \bar L_n$;
\item[(c)]  $r_{n}\restriction\bar L_{n+1}=\id$;
\item[(d)] if $\sigma$ is a maximal simplex of $L_n$ and $\bar\sigma \subset A_n\cup \bar K$, then $\bar\sigma\subset \dom(r_{n})$ and $r_{n}\restriction\bar\sigma=\id$;
\item[(e)] if $\sigma$ is a maximal simplex of $L_n$ and $\bar\sigma \not\subset A_n\cup \bar K$, then $r_{n}(\dom(r_{n}) \cap \bar\sigma)\subset\partial\sigma$;
\item[(f)] $\bar L_{n+1}=\bar L_n\setminus\bigcup\{\sigma^\circ:\sigma$ is a maximal simplex of $L_n$ with $\bar\sigma\not\subset A_n\cup \bar K\}$.
\end{itemize}
Assume that for some $n\in\w$ the subcomplex $L_n$ and the subset $A_n\subset \bar L_n$ have been constructed. Let $\mathcal S$ be the family of simplexes  of the complex $L_n$ such that $\sigma^\circ\not\subset \bar K\cup A_n$. For every simplex $\sigma\in\mathcal S$ fix a point $d_\sigma\in\sigma^\circ\setminus (A_n\cup \bar K)$. The definition of the topology on the complex $\bar L_n$ guarantees that the set $D=\{d_\sigma:\sigma\in\mathcal S\}$ is closed in $\bar L_n$. 
Then $\dom(r_{n})=\bar L_n\setminus D$ is an open subset of $\bar L_n$. Let $\mathcal M$ be the family of all maximal simplexes of the complex $L_n$ 
and observe that $\bar L_n=\bigcup_{\sigma\in\M}\bar\sigma$. Put $\mathcal M'=\{\sigma\in\M:\bar\sigma\cap D\ne\emptyset\}$ and $\bar L_{n+1}=\bar L_n\setminus\bigcup_{\sigma\in\M'}\sigma^\circ$. It follows that $\bar L_{n+1}$ coincides with the geometric realization of some subcomplex $L_{n+1}$ of the simplicial complex $L_n$. 

For every simplex $\sigma\in\M'$ fix any continuous map $r_\sigma:\bar\sigma\setminus D\to\partial\sigma$ such that $r_\sigma|\partial \sigma\setminus D=\id$.
Such a map $r_\sigma$ can be constructed as follows. Take any homeomorphism $h:\bar\sigma\to \bar B$ onto the closed unit ball $\bar B=\{x\in\IR^k:\|x\|\le1\}$  in the Euclidean space $\IR^k$ of dimension $k=\dim(\bar\sigma)=|\sigma|-1$. Under this homeomorphism the (combinatorial) boundary $\partial\sigma$ of $\bar\sigma$ maps onto the unit sphere $S=\{x\in\IR^k:\|x\|=1\}$.
  Next, choose any point $c\in h(D\cap \bar\sigma)$ and consider the retraction $r:\bar B\setminus\{c\}\to S\setminus\{c\}$ assigning to each point $x\in\bar B\setminus\{c\}$ the unique point of the intersection $S\cap (c+(x-c)\IR_+)$ of the sphere $S$ with the ray $c+(x-c)\IR_+$ starting at $c$ and passing through $x$. Then $r_\sigma=h^{-1}\circ r\circ h|\bar\sigma\setminus D$ is a required retraction of $\bar\sigma\setminus D$ onto $\partial\sigma$.

Finally define the map the $r_{n}:\bar L_n\setminus D\to \bar L_{n+1}$ letting $r_{n}|\bar\sigma=\id$ for every $\sigma\in\M\setminus\M'$ and $r_{n}|\bar\sigma\setminus D=r_\sigma$ for every $\sigma\in\M'$. Letting $A_{n+1}=r_{n}(A_n)\subset \bar L_{n+1}$ we complete the inductive step.
\smallskip

  After completing the inductive construction, we obtain a decreasing sequence of subcomplexes $(L_n)_{n\in\w}$ of $L$. Since $L$ is finite-dimensional, this sequence stabilizes, which means that $L_{n+1}=L_n$ and hence $A_n\cup \bar K=\bar L_n$ for some $n\in\IN$. Put $M= L_n=L_{n+1}$. By finite recursion define a sequence of maps $(\tilde r_k:\dom(\tilde r_k)\to \bar M)_{k=0}^n$ such that 
$\dom(\tilde r_n)=\bar L_n$, $\tilde r_n=\id$, and $\dom(\tilde r_{k})=r_{k}^{-1}(\dom(\tilde r_{k+1}))$ and $\tilde r_k=\tilde r_{k+1}\circ r_{k}{\restriction}\dom(\tilde r_k)$ for every $k\in\{n-1,n-2,\dots,0\}$. By induction it can be shown that for every $k\in\{n,\dots,0\}$ the domain $\dom(\tilde r_k)$ of the map $\tilde r_k$ is an open neighborhood of the set $A_k$ in $\bar L_k$ and $\tilde r_k(A_k)=A_{n+1}\subset \bar M$. Then the open neighborhood $U=\dom(\tilde r_0)$ of $A\cup\bar K$ in $\bar L=\bar L_0$ and the map $r=\tilde r_0\colon U\to \bar M$ have the required properties.
\end{proof}

\begin{proposition}\label{real} For any $n$-discrete family $\F$ of closed subsets of a collectively normal space $X$ there is a continuous surjective map $f:X\to \overbar{K}$ onto the geometric realization of some simplicial complex $K$ of dimension $\dim(K)\le n$ such that for every set $F\in\F$ the image $f(F)$ is a subcomplex of $\overbar K$ and for any subfamily $\mathcal E\subset \mathcal F$ we get $f(\bigcap\mathcal E)=\bigcap_{E\in\mathcal E}f(E)$.
\end{proposition}

\begin{proof} Write $\F$ as the union $\F=\bigcup_{i=1}^n\F_i$ of discrete families $\F_i$. 

Consider the family $\bigwedge \F=\{\bigcap\mathcal E:\mathcal E\subset\F\}\setminus\{\emptyset\}$ endowed with the partial order induced by the inclusion relation. The set $X$ belongs to the family $\bigwedge\F$ being the intersection of the empty subfamily $\emptyset\subset\F$. It follows that $X$ is the largest element of the partially ordered set $\bigwedge\F$.

Let $\M_0$ be the set of minimal elements of the partially ordered set $\bigwedge\F$ and for every ordinal $\alpha$ let $\M_\alpha$ be the set of minimal elements of the partially ordered set $\bigwedge\F\setminus\bigcup_{\beta<\alpha}\M_\beta$. 

\begin{claim}\label{cl:M-distinct} For any distinct sets $M,N\in\mathcal M_\alpha$ the intersection $M\cap N$ belongs to $\bigcup_{\beta<\alpha}\mathcal M_\beta$  as long as $M\cap N\ne\emptyset$. 
\end{claim}

\begin{proof} Indeed, assuming that $M\cap N\notin\bigcup_{\beta<\alpha}\mathcal M_\beta$ and taking into account that $M\cap N$ is a proper subset of $M$ or $N$, we would conclude that  $M$ or $N$  does not belong to $\mathcal M_\alpha$.
\end{proof}

Let $\rank(X)$ be the smallest ordinal $\alpha$ for which the set $\M_\alpha$ is empty. The $n$-discreteness of the family $\mathcal F$ implies that $\rank(X)\le n+1$.

Since the family $\bigwedge\F$ is locally finite, for every $\alpha\le\rank(X)$ the sets $X_{<\alpha}=\bigcup_{\beta<\alpha}\bigcup \M_\beta$ and
$X_{\alpha}=\bigcup_{\beta\le\alpha}\bigcup\M_\beta$ are closed in $X$.
Obviously $X_{<\rank(X)}=X$. 

 Define the function $\rank\colon X\to\rank(X)$ letting $\rank^{-1}(\alpha)=X_{\alpha}\setminus X_{<\alpha}$ for every $\alpha<\rank(X)$. Claim~\ref{cl:M-distinct} implies that for every ordinal $\alpha<\rank(X)$ and every $x\in X_{\alpha}\setminus X_{<\alpha}$ there exists a unique set $\mu_\alpha(x)\in\mathcal M_\alpha$ such that $x\in\mu_\alpha(x)$. So, $\mu_\alpha\colon X_{\alpha}\setminus X_{<\alpha}\to\mathcal M_\alpha$ is a well-defined function. Let $\dot\M_\alpha=\mu_\alpha(X_{\alpha}\setminus X_{<\alpha})$ and observe that $X_\alpha=X_{<\alpha}\cup\bigcup\dot\M_\alpha$.
 
\begin{claim}\label{cl:M} Let $\alpha$ be an ordinal, $M\in\dot\M_\alpha$ and $F\in\F$. If $M\cap F\setminus X_{<\alpha}\ne\emptyset$, then $M\subset F$.
\end{claim}

\begin{proof} Assuming that $M\not\subset F$ and taking into account that $M\cap F\setminus X_{<\alpha}\ne\emptyset$, we conclude that $M\cap F\in\bigwedge\F$ is a proper non-empty subset of $M$ and hence $M\cap F\in\M_{<\alpha}$ by the minimality of $M$. Then $M\cap F\subset X_{<\alpha}$ and hence $M\cap F\setminus X_{<\alpha}=\emptyset$, which contradicts our assumption.
\end{proof}
 
 
 

By (finite) induction, for every $\alpha<\rank(X)$ we shall construct an open neighborhood $U_\alpha$ of the set $X_\alpha$ in $X$ and a continuous map $f_\alpha:U_{\alpha}\to \overbar {K_\alpha}$ into the geometric realization of some simplicial complex $K_\alpha$ of dimension $\dim(K_\alpha)\le\alpha$ such that the following condition is satisfied:
\begin{itemize}
\item[$(*_\alpha)$] for every subfamily $\mathcal E\subset\F$ the image $f_\alpha(X_\alpha\cap\bigcap\mathcal E)$ is a subcomplex of $\overbar {K_\alpha}$ such that $f_\alpha(X_\alpha\cap\bigcap\mathcal E)=f_\alpha(U_\alpha\cap\bigcap\mathcal E)=\bigcap_{E\in\mathcal E}f_\alpha(U_\alpha\cap  E)$.
\end{itemize}

For $\alpha=0$ the family $\M_0$ is discrete. By the collective normality of $X$, each set $M\in\M_0$ has an open neighborhood $O_M\subset X$ such that the indexed family $(O_M)_{M\in\M_0}$  is discrete in $X$. For every $M\in\M_0$ consider the open neighborhood $U_M=O_M\setminus\bigcup\{E\in\bigwedge\F:M\cap E=\emptyset\}$. 
 Then $U_0=\bigcup_{M\in\M_0}U_M$ is an open neighborhood of the closed set $X_0=\bigcup\M_0$ in $X$. Let $K_0$ be the 0-dimensional simplicial complex whose set of vertices coincides with $\M_0=\dot\M_0$. Let $f_0:U_0\to \bar K_0$ be the map assigning to each point $X_0=\bigcup\M_0$ the characteristic function $\delta_{M}\in \bar K_0$ for the unique set $M\in\M_0$ whose neighborhood $U_M$ contains $x$. The map $f_0$ is continuous, since $\overbar{K_0}$ is a discrete space and $f_0$ is locally constant.

We claim that the map $f_0$ satisfies the condition $(*_0)$. It suffices to check that $\bigcap_{E\in\mathcal E}f_0(U_0\cap E)\subset f(X_0\cap\bigcap\mathcal E)$ for any subfamily $\mathcal E\subset\F$. Fix any $\delta_M\in\bigcap_{E\in\mathcal E}f_0(U_0\cap E)$ where $M$ is a vertex of the complex $K_0$. Then for every $E\in\mathcal E$ we have $\delta_M\in f_0(U_0\cap E)$ and hence $U_M\cap E\ne\emptyset$. The choice of the neighborhood $U_M$ guarantees that $M\cap E\ne\emptyset$ and then $M\subset E$ by the minimality of $M$. Then $M\subset\bigcap\mathcal E$ and hence $\{\delta_M\}=f_0(M)\subset f_0(X_0\cap\bigcap\mathcal E)$.
\smallskip

Now assume that for some (finite) ordinal $\alpha<\rank(X)$ we have constructed an open neighborhood $U_{\alpha}$ of $X_{\alpha}$ in $X$ and a  map $f_{\alpha}\colon U_{\alpha}\to \overbar{K_{\alpha}}$ into a similicial complex of dimension $\dim(\bar K_\alpha)\le\alpha$ satisfying the condition $(*_{\alpha})$. 

Put $\beta=\alpha+1$. Let us recall that $\dot\M_\beta=\{M\in\M_\beta:M\not\subset X_\alpha\}$. In each set $M\in\dot\M_\beta$ fix a point $z_M\in M\setminus X_\alpha$ and observe that the set $Z_\beta=\{z_M:M\in\dot\M_\beta\}$ is closed in $X$ (by the local finiteness of $\dot\M_\beta$) and $Z_\beta\subset X_\beta\setminus X_\alpha$. 
 
By the normality of $X$ there exists a continuous map $\lambda_\alpha:X\to [0,1]$ such that the preimage $\lambda_\alpha^{-1}(1)$ contains the set $Z_\beta\cup (X\setminus U_\alpha)$ in its interior,  and the preimage $\lambda_\alpha^{-1}(0)$ contains some open neighborhood $W_\alpha$ of $X_\alpha$ in $X$. Let's observe that $W_\alpha\subset U_\alpha$ and $M\setminus W_\alpha\ne\emptyset$ for every $M\in\dot{\M}_\beta$.

Let $L_\beta:=K_\alpha*\dot\M_\beta$ be the joint of the simplicial complex $K_\alpha$ and the $0$-dimensional simplicial complex whose set of vertices coincides with $\dot\M_\beta$. 
 It follows that $\dim(L_{\beta})\le\dim(K_{\alpha})+1\le \alpha+1=\beta$. 
 
By Claim~\ref{cl:M-distinct}, the family of closed sets $(M\setminus W_\alpha)_{M\in\dot\M_\beta}$ is disjoint and, being locally finite, is discrete. By the collective normality  of $X$, there exists a discrete family of open sets $(V_M)_{M\in\dot\M_\beta}$ in $X$ such that $$M\setminus W_\alpha\subset V_M\subset X\setminus\textstyle\bigcup\{F\in\bigwedge\F:F\cap M\setminus W_\alpha=\emptyset\}$$ for all $M\in\dot\M_\beta$.

By the property $(*_\alpha)$, for every $M\in\dot\M_\beta\subset\bigwedge\F$, the sets $f_\alpha(U_\alpha\cap M)$ and $f_\alpha(X_\alpha\cap M)$ coincide with the geometric realization $\overbar{K_M}$ of some subcomplex $K_M$ of $K_\alpha$.

For every $M\in\dot\M_\beta$ consider the map $g_M:W_\alpha\cup V_M\to \overline{K_\alpha*\{M\}}\subset\bar L_\beta$ defined by 
$$g_M(x)=\begin{cases}
f_\alpha(x)&\mbox{if $x\in W_\alpha$};\\
\delta_{M}&\mbox{if $x\in V_M\setminus U_\alpha$};\\
(1-\lambda_\alpha(x)){\cdot}f_\alpha(x)+\lambda_\alpha(x){\cdot}\delta_{M}&\mbox{if $x\in V_M\cap U_\alpha$}.
\end{cases}
$$ 
Using the properties of the function $\lambda_\alpha$, it can be shown that the function $g_M$ is well-defined and continuous.

By Lemma~\ref{reduce}, there exist a subcomplex $L_M$ of $K_\alpha*\{M\}$ with $\bar K_\alpha\cup\{\delta_{M}\}\subset \bar L_M$, an open set $O_M\subset\overline{K_\alpha*\{M\}}$ containing $g_M(M)\cup \bar L_M$, and a function $r_M:O_M\to \bar L_M$ such that 
\begin{itemize}
\item[(a)] $r_M{\restriction}\bar L_M=\id$, 
\item[(b)] $r_M(g_M(M)\cup\bar K_\alpha)=\bar L_M$,
\item[(c)] $r_M(\bar \sigma\cap O_M)\subset \bar\sigma$ for any simplex $\sigma$ of the complex $K_\alpha*\{M\}$.
\end{itemize}
Let $\Lambda_M$ be the subcomplex of $L_M$ whose geometric realization $\bar\Lambda_M$  coincides with the closure of the set $\bar L_M\setminus\bar K_\alpha$ in $\bar L_M$. 
Consider the open neighborhood $U_M=V_M\cap g_M^{-1}(O_M)$ of $M\setminus W_\alpha$ in $X$.

Let $U_\beta=W_\alpha\cup\bigcup_{M\in\dot\M_\beta}U_M$, $K_\beta=K_\alpha\cup\bigcup_{M\in\dot\M_\beta}\Lambda_M$, and $f_\beta:U_\beta\to\bar K_\beta$ be the map defined by $f_\beta{\restriction}W_\alpha=f_\alpha{\restriction}W_\alpha$ and $f_\beta{\restriction}U_M=r_M\circ g_M{\restriction}U_M$. It remains to check that the map $f_\beta:U_\beta\to\bar K_\beta$ has  property $(*_\beta)$. The proof will be divided into a series of claims.

\begin{claim}\label{cl3} $f_\beta{\restriction}M=r_M\circ g_M{\restriction}M$ for every $M\in\dot\M_\beta$.
\end{claim}

\begin{proof} Take any $x\in M$. If $x\notin W_\alpha$, then $f_\beta(x)=r_M\circ g_M(x)$ by the definition of $f_\beta$. If $x\in W_\alpha$, then $g_M(x)=f_\alpha(x)\in \bar K_\alpha\subset \bar L_M$ and  $r_M\circ g_M(x)=g_M(x)=f_\alpha(x)=f_\beta(x)$.
\end{proof}

\begin{claim}\label{cl4} For every $M\in\dot\M_\beta$,
$$\bar K_\alpha\cap \bar \Lambda_M\subset\bar K_M\mbox{ \ and   }f_\beta(M)=r_M(g_M(M))\subset \bar K_M\cup\bar\Lambda_M.$$
\end{claim}

\begin{proof} Since $f_\alpha(U_\alpha\cap M)=\bar K_M$, the definition of the map $g_M$ ensures that $g_M(M)\subset \overline{K_M*\{M\}}$. The property (c) of the retraction $r_M$ guarantees that $r_M(g_M(M))\subset \overline{K_M*\{M\}}$. By the property (b) of the map $r_M$,
$$\bar \Lambda_M\setminus \bar K_\alpha=\bar L_M\setminus \bar K_\alpha=r_M(g_M(M))\setminus\bar K_\alpha\subset \overline{K_M*\{M\}}.$$Then $\bar \Lambda_M=\overline{\bar\Lambda_M\setminus \bar K_\alpha}\subset \overline{K_M*\{M\}}$ and finally
$\bar \Lambda_M\cap\bar K_\alpha\subset \overline{K_M*\{M\}}\cap\bar K_\alpha=\bar K_M.$
By Claim~\ref{cl3},
$$f_\beta(M)=r_M(g_M(M))\subset \bar L_M\cap \overline{K_M*\{M\}}=(\bar K_\alpha\cup\bar\Lambda_M)\cap \overline{K_M*\{M\}}=\bar K_M\cup\bar\Lambda_M.$$
\end{proof}

 \begin{claim}\label{cl5} $f_\beta(M)=\bar K_M\cup \bar \Lambda_M=\bar K_M\cup (\bar \Lambda_M\setminus \bar K_\alpha)$ for every $M\in\dot\M_\beta$.
\end{claim}

\begin{proof}  By Claim~\ref{cl4},
$$\bar K_M=f_\alpha(M\cap X_\alpha)\subset f_\beta(M)\cap\bar K_\alpha\subset (\bar K_M\cup\bar\Lambda_M)\cap\bar K_\alpha=\bar K_M$$and hence $f_\beta(M)\cap \bar K_\alpha=\bar K_M$. Applying the property (b) of the retraction $r_M$ and the inclusion $\bar \Lambda_M\cap\bar K_\alpha\subset\bar K_M$, we obtain the desired equality
\begin{multline*}f_\beta(M)=r_M(g_M(M))=(f_\beta(M)\cap\bar K_\alpha)\cup(r_M(g_M(M))\setminus\bar K_\alpha)=\\=\bar K_M\cup (\bar L_M\setminus \bar K_\alpha)=\bar K_M\cup (\bar\Lambda_M\setminus\bar K_\alpha)=\bar K_M\cup\bar \Lambda_M.
\end{multline*}
\end{proof}

\begin{claim}\label{cl:eq} For every $F\in\bigwedge\F$ we have $$\{M\in\dot\M_\beta:M\subset F\}=\{M\in\dot\M_\beta:F\cap M\setminus X_\alpha\ne\emptyset\}=\{M\in\dot\M_\beta:U_M\cap F\ne\emptyset\}.$$
\end{claim}

\begin{proof} Fix any $M\in\dot\M_\beta$. Claim~\ref{cl:eq} can be derived from the  the following statements.
\smallskip

1. If $M\subset F$, then $\emptyset\ne M\setminus X_\alpha=F\cap M\setminus X_\alpha$.

2. If $F\cap M\setminus X_\alpha\ne\emptyset$, then $M\subset F$ by Claim~\ref{cl:M}. 

3. If $M\subset F$, then $\emptyset\ne M\setminus W_\alpha=F\cap M\setminus W_\alpha\subset F\cap U_M$.

4. If $U_M\cap F\ne\emptyset$, then $\emptyset\ne F\cap M\setminus W_\alpha\subset F\cap M\setminus X_{\alpha}$.
\end{proof}

\begin{claim}\label{cl7} For every $M\in\dot\M_\beta$  and $F\in\bigwedge \F$ we have
$$f_\beta(U_M\cap F)\subset \bar \Lambda_M\cup f_\alpha(X_\alpha\cap F).$$
\end{claim}
 
\begin{proof} The inclusion is trivially true of $U_M\cap F=\emptyset$. So, we assume that $U_M\cap F\ne\emptyset$ and hence $M\subset F$ by Claim~\ref{cl:eq}.

By the condition $(*_\alpha)$ the set $f_\alpha(X_\alpha\cap F)=f_\alpha(U_\alpha\cap F)$ is equal to the geometric realization $\bar K_F$ of some subcomplex $K_F$ of $K_\alpha$. The inclusion $M\subset F$ implies $K_M\subset K_F$.

The definition of the map $g_M$ ensures that $g_M(U_M\cap F)\subset \overline{K_F*\{M\}}$. The property (c) of the retraction $r_M$ guarantees that $r_M\circ g_M(U_M\cap F)\subset \overline{K_F*\{M\}}$.
Then $$f_\beta(U_M\cap F)=r_M\circ g_M(U_M\cap F)\subset \bar L_M\cap \overline{K_F*\{M\}}=\bar K_F\cup\bar\Lambda_M=f_\alpha(X_\alpha\cap F)\cup\bar\Lambda_M.$$
\end{proof}

Taking into account Claim~\ref{cl:eq}, for every $F\in\bigwedge\F$, consider the subfamily
$$\dot\M_F=\{M\in\dot\M_\beta:M\subset F\}=\{M\in\dot\M_\beta:F\cap M\setminus X_\alpha\ne\emptyset\}=\{M\in\dot\M_\beta:U_M\cap F\ne\emptyset\}.$$

\begin{claim}\label{cl:eq2} For every $F\in\bigwedge\F$, the sets $f_\beta(F\cap X_\beta)$ and $f_\beta(F\cap U_\beta)$ are equal to the subcomplex 
$$f_\alpha(F\cap X_\alpha)\cup\bigcup_{M\in\dot\M_F}\bar\Lambda_M=f_\alpha(F\cap X_\alpha)\cup\bigcup_{M\in\dot\M_F}(\bar \Lambda_M\setminus \bar K_\alpha)$$
of $\bar K_\beta$.
\end{claim}

\begin{proof} By the inductive assumption $(*_\alpha)$, the set $f_\alpha(U_\alpha \cap F)=f_\alpha(W_\alpha\cap F)=f_\alpha(X_\alpha\cap F)$ is equal to the geometric relaization $\bar K_F$ of some subcomplex $K_F$ of $K_\alpha$. For every $M\in\dot\M_F$ the inclusion $M\subset F$ implies 
$$\bar K_M=f_\alpha(X_\alpha\cap M)\subset f_\alpha(X_\alpha\cap F)=\bar K_F$$and consequently, $\bar \Lambda_M\cap \bar K_\alpha\subset \bar K_M\subset \bar K_F$.

The choice of the family $\dot\M_F$ ensures that 
$$F\cap X_\beta=(F\cap X_\alpha)\cup\bigcup_{M\in\dot \M_\beta}F\cap M=(F\cap X_\alpha)\cup\bigcup\dot\M_F vc$$
and
$$
F\cap U_\beta=(F\cap W_\alpha)\cup\bigcup_{M\in\dot \M_\beta}(F\cap U_M).
$$
 
By Claim~\ref{cl5}, 
\begin{multline*}
f_\beta(F\cap X_\beta)=f_\beta(F\cap X_\alpha)\cup \bigcup_{M\in\dot\M_F}f_\beta(M)=\bar K_F\cup\bigcup_{M\in\dot\M_F}(\bar K_M\cup\bar \Lambda_M)=\\=\bar K_F\cup\bigcup_{M\in\dot\M_F}\bar \Lambda_M=
\bar K_F\cup\bigcup_{M\in\dot\M_F}(\bar \Lambda_M\setminus \bar K_\alpha)
\end{multline*}
and by Claim~\ref{cl7},
$$
\begin{aligned}
f_\beta(F\cap X_\beta)&\subset f_\beta(F\cap U_\beta)=
f_\beta(F\cap W_\alpha)\cup\bigcup_{M\in\dot M_F}f_\beta(F\cap U_M)=\\
&=
f_\beta(F\cap W_\alpha)\cup\bigcup_{M\in\dot M_F}(f_\alpha(X_\alpha\cap F)\cup\bar\Lambda_M)=\\
&= f_\alpha(F\cap W_\alpha)\cup\bigcup_{M\in\dot\M_F}\bar\Lambda_M=\bar K_F\cup\bigcup_{M\in\dot\M_F}\bar \Lambda_M=f_\beta(F\cap X_\beta),
\end{aligned}
$$
which implies the desired equality
$$
f_\beta(F\cap X_\beta)=f_\beta(F\cap U_\beta)=f_\alpha(F\cap X_\alpha)\cup\bigcup_{M\in\dot\M_F}(\bar \Lambda_M\setminus\bar K_\alpha).
$$
\end{proof}

Finally, we are able to check the condition $(*_\beta)$. Take any subfamily $\mathcal E\subset\F$ and put $F=\bigcap\mathcal E$. Using Claim~\ref{cl:eq}, we can show that $\dot\M_F=\bigcap_{E\in\mathcal E}\dot\M_E$. 
Taking into account that the family $(\bar \Lambda_M\setminus \bar K_\alpha)_{M\in\dot\M_\beta}$ is disjoint, one can check that
$$\bigcap_{E\in\mathcal E}\Big(\bigcup_{M\in\dot\M_E}\bar \Lambda_M\setminus\bar K_\alpha\Big)=\bigcup_{M\in\dot M_F}\bar \Lambda_M\setminus\bar K_\alpha.$$
By Claim~\ref{cl:eq2} and the condition $(*_\alpha)$,
$$
\begin{aligned}
\bigcap_{E\in\mathcal E}f_\beta(U_\beta\cap E)&=\bigcap_{E\in\mathcal E}\big(f_\alpha(X_\alpha\cap E)\cup\bigcup_{M\in\dot\M_E}\bar \Lambda_M\setminus \bar K_\alpha\big)=\\
&=\Big(\bigcap_{E\in\mathcal E}f_\alpha(X_\alpha\cap E)\Big)\cup\Big(\bigcap_{E\in\mathcal E}\bigcup_{M\in\dot\M_E}\bar \Lambda_M\setminus \bar K_\alpha\Big)=\\
&=f_\alpha(X_\alpha\cap F)\cup\Big(\bigcup_{M\in\dot\M_F}\bar \Lambda_M\setminus \bar K_\alpha\Big)=f_\beta(X_\beta\cap F).
\end{aligned}
$$
which completes the proof of the condition $(*_\beta)$ and also completes the inductive step.
\smallskip

After completing the inductive construction, consider the finite ordinal $\alpha=\rank(X)-1\le n$ and observe that $X_\alpha=X$. Put $K=K_\alpha$ and observe that condition $(*_\alpha)$ implies that the map $f=f_\alpha\colon X\to \overbar K$ has the property required in Proposition~\ref{real}.
\end{proof}

\section{A Key Lemma}\label{s:key}

\begin{lemma}\label{key} Let $\F$ be an $n$-discrete family  of closed subsets of a collectively normal space $X$.
For every closed subset $C\subset X$ there exists an $(n+2)$-discrete family $\mathcal G$ of closed subsets of $X$ such that $C=\bigcup\mathcal G$ and for any linked family $\mathcal L\subset \mathcal F\cup\mathcal G$ with $\mathcal L\not\subset\F$ the intersection $\bigcap\mathcal L$ is not empty.
\end{lemma}

\begin{proof} Since the family $\F_C=\F\cup\{C\}$ is $(n+1)$-discrete, by Proposition~\ref{real}, there exists a surjective continuous map $f\colon X\to \overbar{K}$ onto the geometric realization of some simplicial complex $K$ of dimension $\dim(K)\le n+1$ such that for any subfamily $\mathcal E\subset\{C\}\cup\F$ the image $f(\bigcap\mathcal E)$ is a subcomplex of $\overbar K$ equal to the intersection $\bigcap_{E\in\mathcal E}f(E)$.

Let $\mathcal F'=\{f(F):F\in\F\}$. Let $\mathcal G'$ be the family of all closed stars of the second barycentric subdivision of the subcomplex $f(C)$, centered at the vertices of the first barycentric subdivision of $f(C)$. By Proposition~\ref{subdiv}, the family $\mathcal G'$ is an $(n+2)$-discrete closed cover of $f(C)$ such that any linked family $\mathcal L'\subset\F'\cup\mathcal G'$ with $\mathcal L'\cap\mathcal G'\ne\emptyset$ has $\bigcap\mathcal L'\ne\emptyset$.

Consider the family $\mathcal G=\{C\cap f^{-1}(G'):G'\in\mathcal G'\}$ and observe that $\mathcal G$ is an $(n+2)$-discrete closed cover of $C$.
Now take any linked system $\mathcal L\subset\F\cup\mathcal G$ with $\mathcal L\not\subset \F$ and consider the linked system $\mathcal L'=\{f(L):L\in\mathcal L\}\subset \F'\cup{\mathcal G'}$.
It follows that $\mathcal L'\cap\mathcal G'\ne\emptyset$ and hence $\bigcap\mathcal L'$ contains some point $$y\in \bigcap\mathcal L'\subset f(C)\cap\bigcap_{L\in\mathcal L\cap\F}f(L)=f\big(C\cap\textstyle{\bigcap}(\mathcal L\cap\F)\big).$$
Then we can find a point $x\in C\cap\bigcap(\mathcal L\cap\F)$ with $f(x)=y$ and conclude that $x\in\bigcap\mathcal L$.
\end{proof}

\section{Proof of Theorem~\ref{t:main}}\label{s:main}

 Given a collectionwise normal $\aleph$-space $X$, we should prove that $X$ has a binary $\sigma$-discrete closed $k$-network.

The space $X$, being an $\aleph$-space, admits a $\sigma$-discrete closed $k$-network $\mathcal N=\bigcup_{i\in\w}\mathcal N_i$. 

By induction, for every $n\in\w$ we shall construct a $2^{n+1}$-discrete family $\F_n$ of closed subsets of $X$ that has the following property
\begin{itemize}
\item[$(*_n)$]  $\F_n=\bigcup_{N\in\N_n}\F_N$ where  for every $N\in\mathcal N_n$ the subfamily $\mathcal F_N$ has the properties: $N=\bigcup\mathcal F_N$ and each linked family $\mathcal L\subset \mathcal F_N\cup \bigcup_{k<n}\mathcal F_k$ with $\mathcal L\not\subset \bigcup_{k<n}\mathcal F_k$ has $N\cap\bigcap\mathcal L\ne\emptyset$.
\end{itemize}

Put $\mathcal F_0=\N_0$ and observe that the family $\F_0$ is $2$-discrete and has property $(*_0)$. Indeed, for every $N\in\mathcal N_0$ consider the subfamily $\F_N=\{N\}\subset\F_0$ and observe that for any non-empty linked family $\mathcal L\subset\{N\}$ we have $N\cap \bigcap\mathcal L=N\ne\emptyset$.
\smallskip

Now assume that for some $n\in\w$ we have constructed $2^{i+1}$-discrete families $\mathcal F_i$, $i<n$. Then the family $\F=\bigcup_{i<n}\F_i$ is $m$-discrete for $m=\sum_{i<n}2^{i+1}=2^{n+1}-2$.  For every $N\in\mathcal N_{n}$ we can apply Lemma~\ref{key} and find a $(m+2)$-discrete closed cover $\mathcal F_N$ of the set $N$ such that each linked family $\mathcal L\subset \mathcal F_N\cup \bigcup_{k<n}\mathcal F_k$ with $\mathcal L\not\subset \bigcup_{k<n}\mathcal F_k$ has $N\cap\bigcap\mathcal L\ne\emptyset$. Then the $2^{n+1}$-discrete family $\F_n:=\bigcup_{N\in\mathcal N_n}\mathcal F_N$ has the property $(*_n)$.
\smallskip

We claim that $\mathcal F=\bigcup_{n\in\w}\F_n$ is a $\sigma$-discrete binary closed $k$-network for $X$. It is clear that $\mathcal F$ is $\sigma$-discrete and consists of closed sets. Let us show that $\mathcal F$ is binary. Given any finite linked subfamily $\mathcal L\subset\F$, find the smallest number $n\in\w$ such that $\mathcal L\subset\bigcup_{k\le n}\F_k$. Since the family $\mathcal N_n$ is disjoint, there exists a unique $N\in\mathcal N_n$ such that $\mathcal L\subset \F_N\cup\bigcup_{k<n}\mathcal F_k$. The property $(*_n)$ ensures that $\bigcap\mathcal L\ne\emptyset$.
\smallskip

 To see that $\F$ is a $k$-network for $X$, take any open set $U\subset X$ and a compact subset $K\subset U$. Since $\mathcal N$ is a $k$-network, there is a finite subfamily $\mathcal N'\subset\mathcal N$  such that $K\subset \bigcup\mathcal N'\subset U$. By the inductive construction, for every $N\in\mathcal N'$ there exist a number $m_N\in\IN$ and an $m_N$-discrete family $\mathcal D_N\subset\F$ such that $N=\bigcup\mathcal D_N$. By the compactness of $K\cap N$, the $m_N$-discrete subfamily $\K_N=\{D\in \mathcal D_N:D\cap K\ne\emptyset\}$ is finite. Then $\mathcal E=\bigcup_{N\in\mathcal N'}\K_N$ is a finite subfamily of $\F$ such that $K\subset \bigcup\mathcal E\subset U$. This witnesses that $\F$ is a $k$-network for $X$.

\section{Proof of Theorem~\ref{t:super}}\label{s:super}

In this section we prove Theorem~\ref{t:super}. We need to prove that the class of super spaces contains all supercompact spaces, all generalized ordered spaces, all metrizable spaces and all collectionwise normal $\aleph$-spaces. The superness of collectionwise $\aleph$-spaces has been proved in Theorem~\ref{t:main}. Since each metrizable space is a collectionwise normal $\aleph$-space, Theorem~\ref{t:main} implies that metrizable spaces are super. The superness of supercompact spaces is proved in Theorem 2.2 in \cite{BKT}.
\smallskip

It remains to prove the superness of generalized ordered spaces. Let $X$ be a GO-space and $\mathcal B$ be a base of the topology of $X$, consisting of order-convex sets with respect to some linear order $\le $ on $X$. Let us recall that a subset $C\subset X$ is {\em order-convex} if for any points $x,y\in C$ the set $C$ contains the order interval $[x,y]:=\{z\in X:x\le z\le y\}$.

Let $\mathcal K$ be the family of all closed order-convex sets in the GO-space $(X,\le)$. We claim that $\K$ is a binary closed $k$-network for $X$, witnessing that the space $X$ is super. To prove that $\K$ is binary, take any nonempty finite linked family $\mathcal L\subset\K$. For any sets $A,B\in\mathcal L$ choose a point $x_{AB}\in A\cap B$ such that $x_{AB}=x_{BA}$. For every $L\in \mathcal L$ let $a_L:=\min\{x_{LB}:B\in\mathcal L\}$ and $b_L:=\max\{x_{LB}:B\in\mathcal L\}$.
The order-convexity of the set $L$ guarantees that $[a_L,b_L]\subset L$. Let $a:=\max\{a_L:L\in\mathcal L\}$ and $b:=\min\{b_L:L\in\mathcal L\}$. Find sets $A,B\in\mathcal L$ such that $a=a_A$ and $b=b_B$. Then $a=a_A\le x_{AB}=x_{BA}\le b_B=b$ and hence $[a,b]\subset [a_L,b_L]\subset L$ for all $L\in\mathcal L$, which means that $[a,b]\subset \bigcap\mathcal L$ and the intersection $\bigcap\mathcal L$ is not empty. Therefore, the linked family $\mathcal L$ is centered and the family $\mathcal K$ is binary.

It remains to show that $\mathcal K$ is a $k$-network for $X$. Fix any open set $U\subset X$ and any compact set $K\subset U$. By definition, a generalized ordered space is Hausdorff.  This implies (see \cite{Ban}) that the linear order $\le $ is closed in $X\times X$ and order-invervals $[x,y]$ are closed subsets of $X$. By \cite[4.1]{Lutzer}, generalized ordered spaces are regular. This allows us for every $x\in K$ to choose an open order-convex set $O_x\in\mathcal B$ such that $x\in O_x\subset \overline{O_x}\subset U$. The closedness of the linear order on $X$ implies that the closure $\overline{O_x}$ of the order-convex set $O_x$ remains order-convex. By the compactness of $K$, the open cover $\{O_x:x\in K\}$ of $K$ has a finite subcover $\{O_x:x\in F\}$ (here $F$ is a suitable finite subset of $K$). Then the family $\F=\{\overline{O_x}:x\in F\}\subset\mathcal K$ has the property $K\subset\bigcup\F\subset U$, witnessing that $\mathcal K$ is a binary closed $k$-network and the GO-space $X$ is super.

\section{Proof of Theorem~\ref{t:product}}\label{s:product}

Given a discretely dense subset $X$ of the Tychonoff product $\prod_{\alpha\in A}X_\alpha$ of super spaces $X_\alpha$, we should prove that $X$ is super. For every $\alpha\in A$ fix a binary closed $k$-network $\mathcal K_\alpha$ in the super space $X_\alpha$. Replacing $\K_\alpha$ by $\K_\alpha\cup\{X_\alpha\}$, we can assume that $X_\alpha\in\K_\alpha$. For every finite subset $F\subset A$ let $\pr_F:X\to \prod_{\alpha\in F}X_\alpha$, $\pr_F:x\mapsto x{\restriction}_F$,  be the natural projection of $X$ onto the finite product $\prod_{\alpha\in F}X_\alpha$. The discrete density of $X$ in $\prod_{\alpha\in A}X_\alpha$ guarantees that $\pr_F(X)=\prod_{\alpha\in F}X_\alpha$.

Denote by $[A]^{<\w}$ the family of finite subsets of the index set $A$. We claim that the family $$\mathcal K=\bigcup_{F\in[A]^{<\w}}\big\{\pr^{-1}_F\big(\prod_{\alpha\in F}K_\alpha\big):(K_\alpha)_{\alpha\in F}\in\prod_{\alpha\in F}\mathcal K_\alpha\big\}$$is a binary $k$-network for the space $X$.

To see that $\K$ is a $k$-network, take any open set $U\subset X$ and a compact set $K\subset U$. By the definition of the Tychonoff product topology and the compactness of $K$, there exists a finite subset $F\subset A$ and an open set $V\subset \prod_{\alpha\in F}X_\alpha$ such that $K\subset \pr_F^{-1}(V)\subset U$.
It is easy to check that the family 
$$\mathcal K_F=\big\{\prod_{\alpha\in F}K_\alpha:(K_\alpha)_{\alpha\in F}\in\prod_{\alpha\in F}\mathcal K_\alpha\}$$
is a $k$-network for the product $\prod_{\alpha\in F}X_\alpha$. Consequently, there exists a finite subfamily $\F\subset \mathcal K_F$ such that $\pr_F(K)\subset\bigcup\F\subset V$. Then $\mathcal E:=\{\pr^{-1}_F(K):K\in\F\}$ is a finite subfamily in $\K$ such that $K\subset\bigcup\mathcal E\subset U$, witnessing that $\K$ is a closed $k$-network for $X$.
\smallskip

Next, we prove that the family $\K$ is binary. Fix any nonempty finite linked subfamily $\mathcal L\subset \K$. By the definition of the family $\K$, each set  $L\in\mathcal L\subset\K$ is equal to $\pr_{F_L}^{-1}(\prod_{\alpha\in F_L}K_{L,\alpha})$ for some finite set $F_L\subset A$ and an indexed family $(K_{L,\alpha})_{\alpha\in F_L}\in\prod_{\alpha\in F_L}\K_\alpha$. Since $X_\alpha\in \K_\alpha$ for all $\alpha\in A$, we can replace each $F_L$ by the finite set $F=\bigcup_{L\in\mathcal L}F_L$ and assume that $F=F_L$ for all $L\in\mathcal L$. The linkedness of the family $\mathcal L$ implies that for every $\alpha\in F$ the family $\{K_{L,\alpha}\}_{\alpha\in F}\subset \K_\alpha$ is linked and hence centered (by the binarity of $\K_\alpha$).
Then the intersection $\bigcap_{L\in\mathcal L}K_{L,\alpha}$ contains some point $x_\alpha$. Since $X$ is discretely dense in $\prod_{\alpha\in A}X_\alpha$, there exists $x\in X$ such that $x(\alpha)=x_\alpha$ for all $\alpha\in F$. It follows that
$$x\in\pr_F^{-1}(\prod_{\alpha\in F}(\bigcap_{L\in\mathcal L}K_{L,\alpha}))\subset \bigcap_{L\in\mathcal L}\pr_F^{-1}(\prod_{\alpha\in F}K_{L,\alpha})=\bigcap\mathcal L,$$ witnessing that the family $\mathcal L$ is centered. Therefore, the closed $k$-network $\K$ for $X$ is binary.

\end{document}